\documentclass{article}
\usepackage[utf8]{inputenc}
\usepackage[english]{babel}
\usepackage{amsmath}
\usepackage{amssymb}
\usepackage{amsthm}
\usepackage{mathabx}
\usepackage{graphicx}
\usepackage[ruled,linesnumbered]{algorithm2e}
\usepackage{tikz,tikz-3dplot}
\usetikzlibrary{calc}
\usetikzlibrary{decorations.pathmorphing}
\usepackage{pgfplots}
\usepackage{svg}
\usepackage{xcolor}
\usepackage{float}
\usepackage{booktabs,multirow}
\usepackage{amstext} % for \text macro
\usepackage{array} 
\usepackage{hyperref}

\usepackage{titling}
\usepackage{csquotes}

\newtheorem{theorem}{Theorem}
\newtheorem{definition}[theorem]{Definition}
\newtheorem{proposition}[theorem]{Proposition}

\theoremstyle{definition}
\newtheorem{example}[theorem]{Example}

\newcommand{\R}{\mathbb{R}}
\renewcommand{\P}{\mathcal{P}}
\newcommand{\M}{\mathcal{M}}
\newcommand{\transpose}{{\raisebox{\depth}{\footnotesize
      $\intercal$}}} %Trans for matrices
\newcommand{\spc}[1]{%
  \R^{#1}
}
\newcommand{\set}[1]{%
  \left\lbrace {#1}\right\rbrace
}
\newcommand{\setE}[2]{%
  \left\lbrace\, {#1} \,\middle|\, {#2}\,\right\rbrace
}
\newcommand{\idnt}[1]{%
  I_{#1}
}
\newcommand{\zero}{0}

\newcommand{\rep}{re\-pre\-sen\-ta\-tion}
\newcommand{\hrep}{H-\rep}
\newcommand{\vrep}{V-\rep}
\newcommand{\prep}{P-\rep}

\DeclareMathOperator{\conv}{conv}
\DeclareMathOperator{\cone}{cone}
\DeclareMathOperator{\lin}{span}

\DeclareMathOperator{\inter}{int}

\DeclareMathOperator{\dom}{dom}
\DeclareMathOperator{\epi}{epi}

\DeclareMathOperator{\cl}{cl}
\DeclareMathOperator{\recc}{recc}
\DeclareMathOperator{\iconv}{\Box}%
\DeclareMathOperator{\ncone}{N}
\DeclareMathOperator{\lenv}{lenv}

\begin{document}
\title{Calculus of convex polyhedra and polyhedral convex 
functions by utilizing a multiple objective
linear programming solver\textsuperscript{1}}

\author{Daniel Ciripoi\textsuperscript{2}
  \and
  Andreas Löhne\textsuperscript{2}
  \and
  Benjamin Weißing\textsuperscript{2}
}

\footnotetext[1]{This research was supported by the German
Research Foundation (DFG) grant number LO--1379/7--1.}
\footnotetext[2]{Friedrich Schiller University Jena,
Department of Mathematics, 07737 Jena, Germany,
  [daniel.ciripoi\textbar andreas.loehne\textbar
   benjamin.weissing]@uni-jena.de
}
\maketitle

\begin{abstract}
The article deals with operations defined on convex polyhedra or polyhedral
convex functions. Given two convex polyhedra, operations like Minkowski sum,
intersection and closed convex hull of the union are considered. Basic
operations for one convex polyhedron are, for example, the polar,
the conical hull and the image under affine transformation. The concept of a
{\em P-representation} of a convex polyhedron is introduced. It is shown that
many polyhedral calculus operations can be expressed explicitly in terms of
P-representations. We point out that all the relevant computational effort
for polyhedral calculus consists in computing projections of convex polyhedra.
In order to compute projections we use a recent result saying that multiple
objective linear programming (MOLP) is equivalent to the polyhedral projection
problem. Based on the MOLP-solver {\em bensolve} a polyhedral calculus toolbox
for Matlab and GNU Octave is developed. Some numerical experiments are
discussed.

\medskip

\end{abstract}

%%%%%%%%%%%%%%%%%%%%%% 
\section{Introduction}

Convex polyhedra and polyhedral convex functions are relevant in many
disciplines of mathematics and sciences. They can be used to approximate
convex sets and convex functions with the advantage of always having finite
representations. This naturally leads to the need of a calculus, that is, a
collection of operations on one, two or even finitely many convex polyhedra
or polyhedral convex functions.

We introduce a {\em P-representation} of a convex polyhedron, where
the {`P'} stands for {`projection'},
and show that typical calculus operations can be expressed explicitly
in terms of P-representations. It turns out that
all the computational
effort in polyhedral calculus consists in computing \hrep s
(inequality representations) and/or \vrep s (representations
by finitely many points and directions) from \prep s. To this end
we utilize a multiple objective linear programming solver based on
the fact that polyhedral projection is equivalent to multiple
objective linear programming
\cite{equivalence_paper}.

Motivated by the relevance of P-representations for polyhedral calculus,
it appears to be natural to define a convex polyhedron as follows:

\begin{definition}\label{def:polyhedron}
  Given
  matrices $M \in \mathbb{R}^{q \times n}$,
  $B \in \mathbb{R}^{m \times n}$, and vectors
  \[a \in \left( \mathbb{R} \cup \left\{-\infty\right\} \right)^m,\;
  b \in \left( \mathbb{R} \cup \left\{+\infty\right\} \right)^m,\;
  l \in \left( \mathbb{R} \cup \left\{-\infty\right\} \right)^n,\;
  u \in \left( \mathbb{R} \cup \left\{+\infty\right\} \right)^n,\]
  the set 
  \begin{equation}\label{eq:prep}
    P = \setE{M x}{a \leq Bx\leq b,\, l \leq x \leq u}
  \end{equation}
  is called a {\em convex polyhedron\footnote{We solely deal with {\em convex}
  polyhedra in this paper, thus we omit the term `convex' in
  subsequent occurrences.}}.
  The tuple \(\left(M,B,a,b,l,u\right)\) is called
  {\em \prep} of the convex polyhedron $P$. 
\end{definition}

Thus, a polyhedral set is the $q$-dimensional set we obtain by mapping
the points in $\mathbb{R}^n$ which are bounded by the vectors $l$ and
$u$ and satisfy the given system of $2m$ affine inequalities given by
$a,b$ and $B$.  We use the symbol $\emptyset$ to indicate components
that do not occur. For example,
$(M,\emptyset,\emptyset,\emptyset,l,u)$ represents the polyhedron
$P = \setE{M x}{l \leq x \leq u}$ and $(M,B,a,\emptyset,l,\emptyset)$
represents $P = \setE{M x}{a \leq Bx,\; l \leq x}$.

In the literature, see e.g. \cite{rockafellar}, a polyhedron
is usually defined as an intersection of finitely many closed
half-spaces, which refers to the special case of $M$ being the
unit matrix and thus $q=n$. In this case, we have 
    \begin{equation}\label{eq:hrep}
      P = \setE{x}{a \leq Bx\leq b,\, l \leq x \leq u}
    \end{equation}
and the tuple \(\left(B,a,b,l,u\right)\) is called
{\em H-representation} of $P$. 

A simple reformulation shows that \eqref{eq:prep}
is the {\em projection} of an H-re\-pre\-sented
polyhedron
  \begin{equation}\label{eq:q}
    Q = \setE{(x,y) \in \R^n \times \R^q}{y = Mx,\, a \leq Bx\leq b,\,
	l \leq x \leq u}
  \end{equation}
onto the $y$-components, which motivates the term {\em P-representation}.

Fourier-Motzkin-Elimination, see e.g. \cite{lauritzen}, provides a tool
for eliminating the $x$-components in the following reformulation
of \eqref{eq:prep}:
\begin{equation}\label{eq:prep1}
  P = \setE{y \in \R^q}{\exists x \in \R^n:\; y = Mx,\, a \leq Bx\leq b,\,
  l \leq x \leq u}\text{.}
\end{equation}
This means that every polyhedron as defined in Definition \ref{def:polyhedron}
admits an \hrep , which justifies to define a polyhedron in an
alternative way via a \prep .

It should be pointed out that in contrast to \eqref{eq:hrep} the
H-representation of a polyhedron is usually defined as
$P=\left\{x \mid Bx \leq b\right\}$ in the literature.  The seemingly
redundant form we use in \eqref{eq:hrep}, where we distinguish between
constraints and variable bounds and where double bounds are employed,
resembles the input format of \emph{bensolve tools}.  Using this more
explicit formulation, users can specify polyhedra directly as they
appear in their respective applications without the need to
reformulate the defining system in the form $Bx \leq b$.

According to the Minkowski-Weyl theorem, every polyhedron admits a
representation in terms of points and directions: Consider a
polyhedron \(P\subseteq\spc{q}\) and let matrices $V \in
\mathbb{R}^{q \times r}$ ($r \geq 1$), $D\in \mathbb{R}^{q \times s}$
($s\geq 0$) and $L \in \mathbb{R}^{q \times t}$ ($t\geq 0$) be given,
where we write 
$D=\emptyset$ and $L=\emptyset$ if $s=0$ and $t=0$, respectively. The
vectors \(v^1,\ldots,v^r\),
\(d^1,\ldots,d^s\) and \(l^1,\ldots,l^s\) shall denote the columns of $V$,
$D$ and $L$, respectively.
If
\begin{equation}\label{eq:vrep}
  P = \conv\set{v^1,\ldots,v^r} + \cone\set{d^1,\ldots,d^s} +
  \lin \set{l^1,\ldots,l^s}
\end{equation}
holds, where we set $\cone \emptyset = \set{0}$ and
$\lin \emptyset = \set{0}$,
then \(\left(V,D,L\right)\) is called {\em
\vrep} of $P$.

Given a \vrep\ $(V,D,L)$ of a polyhedron, it is evident that
\begin{equation*}\label{eq:v2p}
\left(
  (V,D,L),
  (1_{(r)}^\transpose,0_{(s+t)}^\transpose),
  1,
  1,
\left(\begin{array}{c}
  0_{(r+s)}\\
  -\infty_{(t)}
  \end{array}\right),
  \infty_{(r+s+t)}
  \right)
\end{equation*}
is a P-representation of $P$, where $a_{(n)}$ stands for an $n$-dimensional
column vector all the $n$ components of which equal to $a$.

The problem of computing a \vrep\ for a polyhedron given in
\hrep\ is called the {\em vertex enumeration
problem}. The reverse problem is called {\em facet enumeration
problem} and can be interpreted as vertex enumeration problem under
polarity. Since an \hrep\ is a special case of a \prep ,
vertex enumeration can be seen as special case of the polyhedral
projection problem \cite{diss_benni}, which is, roughly speaking,
the problem to
compute a \vrep\ from a \prep . Analogously, the dual polyhedral projection
problem as introduced in \cite{diss_benni} covers facet enumeration.

An idea related to the one presented here is used in
\cite{matrix_cones} for treating combinatorial optimization problems.
There, the authors make use of the fact that a high~dimensional
polyhedron may have a simple structure, while some low~dimensional
projection may become quite complex.\par

Another polyhedral calculus toolbox, which also covers non-convex
polyhedra, is MPT3 \cite{MPT3}. In Section \ref{sec:num} we compare
MPT3 to our approach.

This article is organized as follows. In Section \ref{sec:polyh}
we present the results about calculus of polyhedral sets in terms
of P-representations. In Section \ref{sec:polyf}, we extend the results to
polyhedral functions. Section \ref{sec:proj} discusses how a MOLP solver 
can be utilized to compute H-representations and V-representations from 
P-representations. Section \ref{sec:num} introduces the polyhedral
calculus software {\em bensolve~tools} by discussing some numerical
experiments.

\section{Polyhedral set calculus via P-representations}\label{sec:polyh}

The notion of \prep\ as introduced in Definition \ref{def:polyhedron} allows
the explicit expression of several polyhedral calculus operations. 
We start this section by listing some results that can be proven easily by 
employing the corresponding definitions. 

\begin{proposition}[Minkowski sum] \label{prop:sum}
  Let the two polyhedra \(A^1\subseteq\spc{q}\) and
  \(A^2\subseteq\spc{q}\) with \prep s \((M^1,B^1,a^1,b^1,l^1,u^1)\)
  and \((M^2,B^2,a^2,b^2,l^2,u^2)\), respectively, be given.  Then the
  sum
  \[
    A^1 + A^2 = \setE{y\in\spc{q}}{\exists y^1 \in A^1, y^2 \in
      A^2\colon y = y^1 + y^2}
  \]
  has the \prep
  \[
    \left(
      \begin{pmatrix}
        M^1 &M^2
      \end{pmatrix},
      \begin{pmatrix}
        B^1 &\zero\\
        \zero &B^2
      \end{pmatrix},
      \begin{pmatrix}
        a^1\\
        a^2
      \end{pmatrix},
      \begin{pmatrix}
        b^1\\
        b^2
      \end{pmatrix},
      \begin{pmatrix}
        l^1\\
        l^2
      \end{pmatrix},
      \begin{pmatrix}
        u^1\\
        u^2
      \end{pmatrix}
    \right)\text{.}
  \]
\end{proposition}

\begin{proposition}[intersection] \label{prop:intsec}
  Let the polyhedra \(A^1\subseteq\spc{q}\) and
  \(A^2\subseteq\spc{q}\) with \prep s
  \((M^1,B^1,a^1,b^1,l^1,u^1)\) and \((M^2,B^2,a^2,b^2,l^2,u^2)\),
  respectively, be given.  Then the intersection
  \(A^1 \cap A^2\subseteq\spc{q}\) is a polyhedron with \prep
\[
  \left(
    \begin{pmatrix}
      M^1 &\zero
    \end{pmatrix},
    \begin{pmatrix}
      B^1 &\zero\\
      \zero &B^2\\
      M^1 &-M^2
    \end{pmatrix},
    \begin{pmatrix}
      a^1\\
      a^2\\
      \zero
    \end{pmatrix},
    \begin{pmatrix}
      b^1\\
      b^2\\
      \zero
    \end{pmatrix},
    \begin{pmatrix}
      l^1\\
      l^2
    \end{pmatrix},
    \begin{pmatrix}
      u^1\\
      u^2
    \end{pmatrix}
  \right)\text{.}
\]
\end{proposition}

\begin{proposition}[Cartesian product] \label{prop:cart}
    Let the polyhedra \(A^1\subseteq\spc{p}\) and
    \(A^2\subseteq\spc{q}\) with \prep s
    \((M^1,B^1,a^1,b^1,l^1,u^1)\) and \((M^2,B^2,a^2,b^2,l^2,u^2)\),
    respectively, be given. Then the Cartesian product
    \(A^1 \times A^2\subseteq\spc{p+q}\) is a polyhedron with \prep
  \[
    \left(
      \begin{pmatrix}
        M^1 &\zero\\
		\zero &M^2
      \end{pmatrix},
      \begin{pmatrix}
        B^1 &\zero\\
        \zero &B^2
      \end{pmatrix},
      \begin{pmatrix}
        a^1\\
        a^2
      \end{pmatrix},
      \begin{pmatrix}
        b^1\\
        b^2
      \end{pmatrix},
      \begin{pmatrix}
        l^1\\
        l^2
      \end{pmatrix},
      \begin{pmatrix}
        u^1\\
        u^2
      \end{pmatrix}
    \right)\text{.}
  \]
\end{proposition}

\begin{proposition}[recession cone] \label{prop:recc}
  Consider the polyhedron
  \(A\subseteq\spc{q}\) with \prep\  \((M,B,a,b,l,u)\). The 
  recession cone of $A$,
  \begin{align*}
    \recc A=\left\{ y \in \R^q \; | \; \forall
    \, x \in A,\, \forall \,t \geq 0 \; : \; x+ty\in A  \right\}
  \end{align*}
  has the \prep\ \((M,B,0 \cdot a,0 \cdot b,0 \cdot l,0 \cdot u)\), where
  we set $0 \cdot \pm \infty := \pm \infty$.
\end{proposition}

We close the first part of this section by an example for the Minkowski sum.
\begin{example}
  Let \(A^1\subseteq\spc{q}\) be the unit ball of the 1-norm,
  \begin{equation*}%\label{eq:a1}
    A^1 = \setE{y \in \spc{q}}{\sum_{i=1}^q \lvert y_i \rvert \leq
      1}\text{,}
  \end{equation*}
  and let
  \[
    A^2 = \setE{y \in \spc{q}}{-1 \leq y \leq 1}
  \]
  be the unit ball of the $\infty$-norm in \(\spc{q}\).  Both norms
  are polyhedral, thus \(A^1\) and \(A^2\) are polyhedra.  
%  By
%  introducing auxiliary variables \(x,z \in \R^q_+\), 
%  we obtain a projective representation of $A_1$:
  Since the set \(A^1\) is the convex hull of the unit
  vectors and their negatives, it can be expressed as 
  \begin{align*}
    A^1 &= \setE{y \in \spc{q}}{\sum_{i=1}^q \lvert y_i \rvert \leqslant
    1}\\
    &= \setE{x - z}{x,z\in \R^q_+ ,
    \sum_{i=1}^q x_i + z_i = 1}\text{.}
  \end{align*}
  This produces a \prep\ which is given by
  \[
    \left(
      \begin{pmatrix}
        \idnt{q} &-\idnt{q}
      \end{pmatrix},
        1^\transpose_{(2q)},
      1,
      1,
      0_{(2q)},
      \infty_{(2q)}
    \right)\text{,}
  \]
  where $\idnt{q}$ denotes the $q \times q$ unit matrix.
  Polyhedron \(A^2\) admits the \prep\
  \[(\idnt{q},\emptyset,\emptyset,\emptyset,-1_{(q)},1_{(q)})\text{.}\]
  Therefore, by
  Proposition \ref{prop:sum}, we obtain the \prep
  \[
    \left(
      \begin{pmatrix}
        \idnt{q} &-\idnt{q} &\idnt{q}
      \end{pmatrix},
      (1^\transpose_{(2q)},0^\transpose_{(q)}),
     1,1,
      \begin{pmatrix}
        \zero_{(2q)}\\
        -1_{(q)}
      \end{pmatrix},
      \begin{pmatrix}
        \infty_{(2q)}\\
        1_{(q)}
      \end{pmatrix}
    \right)
  \]
  for the polyhedron \(A^1 + A^2\).  We illustrate the sets $A^1$,
  $A^2$ and their Minkowski sum $A^1+A^2$ for $q=3$ in
  \autoref{fig:minkowski_sum}.
  \begin{figure}
    \centering
    \includegraphics[scale=0.55,trim=4.3cm 2cm 3.5cm 2cm, clip]{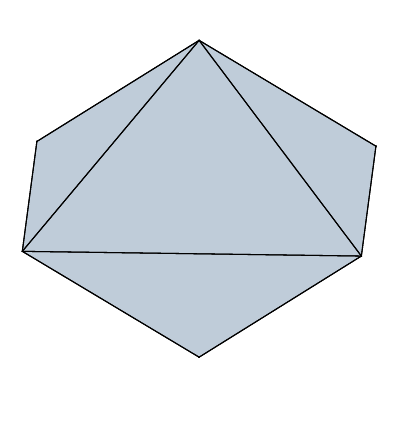}%
    \includegraphics[scale=0.34,trim=1cm 0cm 1.8cm 1cm, clip]{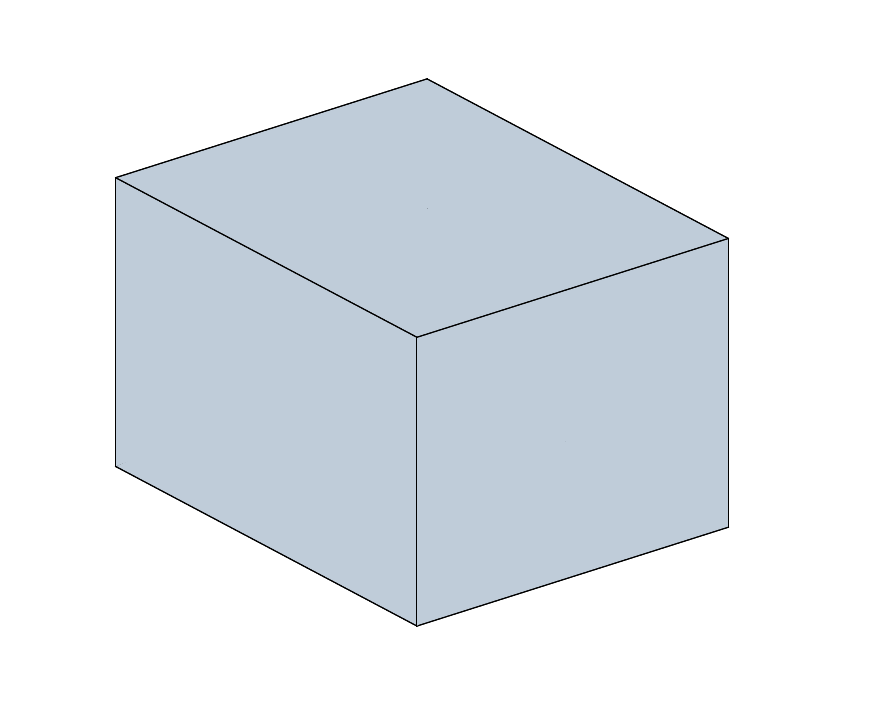}%
    \includegraphics[scale=0.48,trim=3.6cm 2.5cm 2cm 2cm, clip]{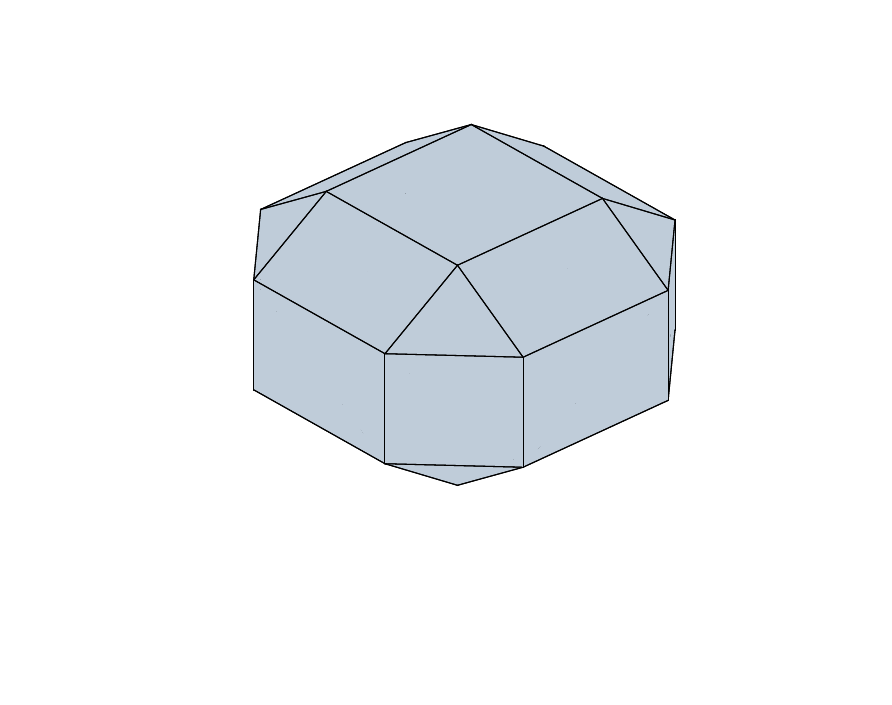}%
    \caption{\label{fig:minkowski_sum}%
      From left to right: 3-dimensional unit ball of 1-norm,
      3-dimensional unit ball of $\infty$-norm, Minkowski sum of the
      unit balls.}%
  \end{figure}
\end{example}

Other polyhedral calculus operations like computing the {\em polar}
of a nonempty polyhedron or computing the {\em closed convex hull of the
union} of finitely many polyhedra can be obtained by utilizing a variant
of the Farkas lemma like Motzkin's transposition theorem. Also in these
cases, the resulting \prep s only require transposition and
rearrangement of the given data.

\begin{theorem}[Motzkin's transposition theorem]
  \label{thm:motzkin_transposition}
  The linear system
  \begin{align*}
    B_1x < b_1, \hspace{2cm} B_2x \leq b_2, \hspace{2cm} B_3x=b_3
  \end{align*}
  has no solution if and only if there exist $z_0 \in \mathbb{R}$ and
  vectors $z_1,z_2,z_3$ such that
  \begin{align*}
    z_1^\intercal B_1 +z_2^\intercal B_2 +z_3^\intercal B_3 &= 0, \\
    z_0 +z_1^\intercal b_1 +z_2^\intercal b_2 +z_3^\intercal b_3 &=0,\\
    \begin{pmatrix}
      z_0\\z_1\\z_2
    \end{pmatrix} \geq 0 , \hspace{0.5cm} \begin{pmatrix}
      z_0\\z_1
    \end{pmatrix} &\neq 0.
  \end{align*}
\end{theorem}
\begin{proof}
	See, for instance, {\cite[Theorem 3.17 / Section 3.3]{gueler}}.
\end{proof}

\begin{proposition}\label{prop:polar}
  The polar 
  \begin{align*}
    A^\circ=\left\{ y \in \mathbb{R}^q \;  | \;
    \forall \,  v \in A \; : \;  y^\intercal v \leq 1 \right\} \,.
  \end{align*}
   of a nonempty polyhedron $A \subseteq \spc{q}$ with \prep\ $(M,B,a,b,l,u)$,
   where $B \in \R^{m\times n}$, 
   has the \prep\ 
  \begin{align*}
    \left( \begin{pmatrix}
      0_{(q\times k)} \;\idnt{q}
    \end{pmatrix} ,
	B',
    \begin{pmatrix}
      0_{(n)} \\ - \infty
    \end{pmatrix}, 
    \begin{pmatrix}
      0_{(n)} \\ 1
    \end{pmatrix},
    \begin{pmatrix}
      0_{(k)} \\ - \infty_{(q)}
    \end{pmatrix},
      \infty_{(k+q)}
    \right)\text{,}
  \end{align*}
  where $B' \in \spc{(n+1)\times k}$ results from 
  \begin{align*}
    \begin{pmatrix}
      B^\transpose & -B^\transpose & \idnt{n} &
	   -\idnt{n} & -M^\transpose \\
      b^\intercal  & -a^\intercal & u^\intercal &
	   -l^\intercal & 0^\intercal_{(q)}
    \end{pmatrix}
  \end{align*}
  by deleting all columns with infinite
  entries.
\end{proposition}
\begin{proof}
  We have $y \in A^\circ$ if and only if
  \begin{align*}
 v \in A \; \Rightarrow \; y^\intercal v \leq 1 \text{.}
  \end{align*}
This is equivalent to the following system being inconsistent:
  \begin{align}\label{eq:sys1}
    v=Mx,\; a \leq Bx \leq b,\; l \leq x \leq u,\;
    y^\intercal v > 1 \text{.}
\end{align}
By Theorem \ref{thm:motzkin_transposition}, \eqref{eq:sys1} has no
solution if and only if
\begin{equation}\label{eq:sys2}
  \begin{aligned}
    \begin{pmatrix} 0_{(n)} \\ -y \end{pmatrix} z_1   +\begin{pmatrix}
    B^\transpose & -B^\transpose & \idnt{n} & -\idnt{n} \\ 0 & 0& 0& 0 \end{pmatrix}
    z_2 +\begin{pmatrix} -M^\transpose \\ \idnt{q} \end{pmatrix} z_3  &=0\\
    z_0   -z_1    +\begin{pmatrix} b^\intercal & -a^\intercal &
    u^\intercal & -l^\intercal \end{pmatrix} z_2 +0_{(q)}^\intercal z_3  &=0 \\
    \begin{pmatrix}
      z_0\\z_1\\z_2
    \end{pmatrix} \geq 0 , \hspace{0.5cm} \begin{pmatrix}
      z_0\\z_1
    \end{pmatrix} &\neq 0
  \end{aligned}
\end{equation}
has a solution, where we assume that the $\pm\infty$-components of the vector
before $z_2$  and the corresponding columns of the matrix before $z_2$
have been deleted.

Consider the case where $z_1\neq 0$. Then, without loss of generality,
we can assume $z_1 = 1$ and by obtaining $y=z_3$ from the second row of the first equation in system \eqref{eq:sys2} respective system can be expressed as
  \begin{align*}
    \begin{pmatrix}
      0_{(n)} \\ - \infty
    \end{pmatrix} \leq \begin{pmatrix} B^\transpose & -B^\transpose
    & \idnt{n} & -\idnt{n} & -M^\transpose \\b^\intercal & -a^\intercal &
    u^\intercal & -l^\intercal & 0_{(q)}^\intercal
    \end{pmatrix} \begin{pmatrix}
      z_2 \\y
    \end{pmatrix} &\leq  \begin{pmatrix}
      0_{(n)} \\1
    \end{pmatrix}\\
	z_2 &\geq 0 \text{,}
  \end{align*}
  which yields the P-representation claimed for $A^\circ$.

  Finally we show that the case $z_1=0$ cannot occur. Otherwise, we get
  $z_0 > 0$ and $z_3=0$. Hence, system (\ref{eq:sys2})
  turns into
  \begin{align*}
    \begin{pmatrix} B^\transpose & -B^\transpose
    & \idnt{n}  & -\idnt{n}  \end{pmatrix} z_2 &=0 \\
    \begin{pmatrix} b^\intercal & -a^\intercal &
     u^\intercal & -l^\intercal  \end{pmatrix} z_2 &<0  \\
	 z_2 &\geq 0\text{.}
  \end{align*}
  As $A$ is assumed to be nonempty, we can choose
  an element $\bar{x}$ such that 
  \begin{align*}
    \bar{x}^\intercal \begin{pmatrix} B^\transpose &
    -B^\transpose & \idnt{n} & -\idnt{n} \end{pmatrix} \leq \begin{pmatrix}
     b^\intercal & -a^\intercal & u^\intercal & -l^\intercal
     \end{pmatrix}\text{.}
  \end{align*}
  This leads to the contradiction 
  \begin{align*}
    0= \bar{x}^\intercal \begin{pmatrix} B^\transpose &
   -B^\transpose & \idnt{n}  & -\idnt{n}  \end{pmatrix} z_2 \leq \begin{pmatrix}
    b^\intercal & -a^\intercal & u^\intercal &
   -l^\intercal \end{pmatrix} z_2 < 0\text{,}
  \end{align*}
  which completes the proof.
\end{proof}

\begin{proposition}\label{prop:polarcone}
  The polar cone 
  \begin{align*}
    A^*=\left\{ y \in \mathbb{R}^q \;  | \;
    \forall \,  x \in A \; : \;  y^\intercal x \leq 0 \right\}
  \end{align*}
   of a nonempty polyhedron $A \subseteq \spc{q}$ with \prep\ $(M,B,a,b,l,u)$
   has the \prep\ 
   \begin{align*}
     \left( \begin{pmatrix}
       0_{(q\times k)} \;\idnt{q}
     \end{pmatrix} ,
 	B',
     \begin{pmatrix}
       0_{(n)} \\ - \infty
     \end{pmatrix}, 
       0_{(n+1)},
     \begin{pmatrix}
       0_{(k)} \\ - \infty_{(q)}
     \end{pmatrix},
       \infty_{(k+q)}
     \right)\text{,}
   \end{align*}
   where $B'$ is the same matrix as in Proposition \ref{prop:polar}. 
\end{proposition}
\begin{proof} Similar to the proof of Proposition \ref{prop:polar}.	
\end{proof}

We now recall some well-known results on polyhedra which allow us to
express further polyhedral calculus operations in terms of
\prep s. Let $A^{\circ \circ}=\left( A^\circ \right)^\circ$ and
$A^{**}=\left( A^* \right)^*$ denote the {\em bipolar}
and the {\em bipolar cone} of a polyhedron $A$. 

\begin{proposition}[see e.g.\ {\cite{rockafellar}}]
  For nonempty polyhedra $A,A_1,A_2 \subseteq \mathbb{R}^q$ and nonempty 
  polyhedral cones $C,C_1,C_2 \subseteq \mathbb{R}^q$ one has
  \begin{enumerate}
  \item $A^{\circ \circ}= \cl \conv \left\{ A \cup
   \left\{ 0 \right\} \right\}$,
  \item If $0 \in A$, then $A^{\circ \circ} =A$,
  \item $C^{**}=C$,
  \item $\cl \cone A =A^{**}$,
  \item If $0 \in A_1 \cup A_2$, then $\cl \conv 
    \left( A_1 \cup A_2 \right)=\left( \left( A_1 \right)^\circ
    \cap \left( A_2 \right)^\circ \right)^\circ$,
  \item $C_1+C_2=\left( \left( C_1 \right)^\circ \cap
    \left( C_2 \right)^\circ \right)^\circ$.
  \end{enumerate}
\end{proposition}

Note that the closure operation cannot be omitted in the polyhedral case 
(take for instance $A=\set{x \in \spc{2} \mid x_2=1}$ in (i)).

As a consequence of the preceding proposition we are able to derive a
P-representation of the closed conic hull of a polyhedron by applying the
polar cone operation twice, see (iv). Furthermore, we obtain
a P-representation of the closed convex hull of the union of two polyhedra
((v) and translation). The {\em normal cone} of a polyhedron
$A \subseteq \R^q$ at a point $\bar x \in A$ is defined as the set
\begin{align*}
  \ncone_{A}(x_0)=\left\{ y \in \R^q \; |
  \; \forall\, x \in A \; : \; y^\intercal(x-\bar x) \leq 0 \right\} \text{.}
\end{align*}
It is known (see e.g. \cite{rockafellar}) that,
  \begin{align*}
    \ncone_{A}(\bar x)=  \left( \cone \left(
    A - \left\{ \bar x \right\} \right) \right)^\circ \text{.}
  \end{align*}
Thus, a \prep\ of $\ncone_{A}(\bar x)$ is obtained by combining some
of the previous results.

\section{Calculus of polyhedral convex functions}\label{sec:polyf}

A function
$f \colon \R^n \rightarrow \R \cup \left\{+ \infty \right\}$ is called
{\em polyhedral} if its epigraph
\begin{align*}
  \epi f = \left\{ (x,r) \in \R^n \times \R \; |
  \; f(x) \leq r \right\}
\end{align*}
is a polyhedron. Since all polyhedra in this article are convex,
polyhedral functions are convex, too.
The {\em domain} of $f$ is defined as
  \begin{align*}
	  \dom f = \{x \in \R^n \mid f(x) < +\infty\}\text{.}
  \end{align*} 

A polyhedral function can be represented by a \prep\ of its epigraph.
Well-known results from Convex Analysis provide 
the relationship between function operations and corresponding
epigraph operations. Thus, using our calculus for polyhedral sets
applied to the epigraphs, we can easily derive calculus operations
for polyhedral functions. If a polyhedral function $f$ 
is represented by a \prep\ $(M,B,a,b,l,u)$ of its epigraph,
then computing a function value $f(x)$ for some $x \in \R^n$
requires to solve the linear program
\begin{align*}
	\min_{r,z} \, r \quad\text{s.t.}\quad \begin{pmatrix}
		x \\ r
	\end{pmatrix}=Mz,\; a \leq B z \leq b,\; l \leq z \leq u\text{.}
\end{align*}

The next four statements combined with the results of Section \ref{sec:polyh}
provide some first calculus operations for polyhedral functions. 
Proofs can be found in Convex Analysis books such as \cite{rockafellar}.

\begin{proposition}
  Let
  $f_1,\ldots,f_k \colon \R^n \rightarrow \R \cup \left\{ + \infty
  \right\}$ be polyhedral functions. The epigraph of the
  pointwise maximum function $\max \left( f_1,\ldots,f_k \right)$
  is 
  \begin{align*}
    \epi \max \left( f_1,\ldots,f_k \right) = \bigcap_{i=1}^k \epi f_i\,.
  \end{align*} 
\end{proposition}  

The {\em lower closed convex envelope} $\lenv(f_1,\ldots,f_k)$
of given polyhedral functions
$f_1,\ldots,f_k \colon \R^n \rightarrow \R \cup
\left\{ +\infty\right\}$ is defined as the largest closed convex function
from $\R^n$ to $\R \cup \left\{ +\infty
\right\}$ majorized by all given functions.

\begin{proposition}
  Let
  $f_1,\ldots,f_k \colon \R^n \rightarrow \R \cup \left\{ + \infty
  \right\}$ be polyhedral functions. The epigraph of the lower
  closed convex envelope function $\lenv \left( f_1,\ldots,f_k \right)$ is
  \begin{align*}
    \epi \lenv \left( f_1,\ldots,f_k \right) =
    \cl\conv \left( \bigcup_{i=1}^k \epi f_i \right)\,.
  \end{align*} 
\end{proposition}

The {\em infimal convolution} of polyhedral functions
  $f_1,\ldots,f_k \colon \R^n \rightarrow \R \cup \left\{ + \infty
  \right\}$  is defined as
\begin{align*}
  \left( f_1 \iconv \dots \iconv f_k  \right)(x) =
  \inf \left\{ f_1(x^1)+ \dots + f_k(x^k) \;|\;
  x^1 + \dots + x^k = x \right\}\text{.}
\end{align*}

\begin{proposition}
  Let
  $f_1,\ldots,f_k \colon \R^n \rightarrow \R \cup \left\{ + \infty
  \right\}$ be polyhedral functions. Then we have
  \begin{align*}
    \epi \left( f_1 \iconv \dots \iconv f_k \right) =
    \epi f_1 + \dots + \epi f_k \text{.}
  \end{align*}
\end{proposition}

\begin{proposition}
  Let
  $f_1,\ldots,f_k \colon \R^n \rightarrow \R \cup \left\{ + \infty
  \right\}$ be polyhedral functions. For the pointwise sum function
  $f_1+\ldots+f_k$, one has 
  \begin{align*}
    \epi \left(f_1+\ldots+f_k \right) = \left\{\left(x , r \right)
	\in \R^n \times \R \;\bigg|\;
   \sum_{i=1}^k r_i=r,\; (x,r_i) \in \epi f_i,\; i=1,\ldots,k \right\} \text{.}
  \end{align*}
\end{proposition}

The {\em conjugate} $f^*:\R^n \rightarrow \R \cup \left\{ + \infty
  \right\}$ of a polyhedral convex
function $f: \R^n \rightarrow \R \cup \left\{ + \infty
  \right\}$ with $\dom f \neq \emptyset$ is defined as
\begin{align*}
  f^*(x^*) = \sup_{x \in \dom f} \{ x^\intercal x^* -f(x) \}\text{.}
\end{align*}
The following result tells us that a \prep\ of the the epigraph of $f^*$ 
can be obtained from a \prep\ of the epigraph of $f$ by polyhedral calculus
operations as discussed in Section \ref{sec:polyh}. 

\begin{proposition}
  Let $f \colon \R^n \rightarrow \R \cup \left\{ +\infty \right\}$ be
  a polyhedral function with $\dom f \neq \emptyset$. Then,
  \begin{align*}
    \epi f^* = \left\{ (x^*,r^*) \; | \; (x^*,-1,r^*) \in  K(f)^* \right\} \,,
  \end{align*}
  where $K(f)^*$ is the polar cone of the polyhedron
  \begin{align*}
    K(f) = \left\{ (x,r,-1) \; | \; (x,r) \in \epi f \right\}\,.
  \end{align*}
\end{proposition}

\begin{proof}
  The epigraph of $f^*$ is the set
  \begin{align*}
    \epi f^* &=\left\{ (x^*,r^*) \; | \; f^*(x^*) \leq r^* \right\} \\
             &=\left\{ (x^*,r^*) \; | \;  \sup_{x \in \dom f}
               \{x^\intercal x^* -f(x)\} \leq r^*\right\}\\
             &=\left\{ (x^*,r^*) \; | \;  \forall \, x \in \dom f
                \; : \; x^\intercal x^* -f(x) \leq r^*\right\}\\
             &=\left\{ (x^*,r^*) \; | \; \forall \, (x,r) \in
               \epi f \; : \; x^\intercal x^* -r \leq r^* \right\}\\
             &=\left\{ (x^*,r^*) \; | \; \forall \, (x,r,-1) \in K(f) \; : \; 
                x^\intercal x^* +r (-1) + (-1) r^* \leq 0 \right\}\\
             &=\left\{ (x^*,r^*) \; | \; (x^*,-1,r^*)
                \in  K(f)^* \right\} \text{,}
  \end{align*}
  which completes the proof.
\end{proof}

Further operations for polyhedral functions can
be obtained in a similar manner. 
For instance, the ability to compute the normal cone of
a P-represented epigraph of
a polyhedral function $f$ can be used to compute a \prep\ of the
(convex) subdifferential of $f$ at a point $x \in \dom f$
by using the well-known formula
\[ \partial f(x) = \left\{ y \in \R^n \;\bigg|\; \begin{pmatrix}y \\ -1 
\end{pmatrix} \in \ncone_{\epi f}\begin{pmatrix}
	x\\f(x)
\end{pmatrix}\right\} \text{.}\]

\section{Computing projections via MOLP}\label{sec:proj}

In this section we briefly outline how to compute the \vrep\ and \hrep\ of a polyhedron given in \prep. For further details concerning the theoretical background the reader is referred to \cite{equivalence_paper,diss_benni}. The main idea is to solve a multiple objective linear program (MOLP) connected to the projection problem. We can then derive the \vrep\ and \hrep\ of the projected polyhedron from the primal and dual solution of the corresponding MOLP.

In multiple objective linear programming , the so-called
upper image plays an important role. The upper image of 
\begin{align}\label{molp} \tag{MOLP}
	\min M x \quad \text{s.t.}\quad a \leq B x \leq b,\; l \leq x \leq u
\end{align}
is the polyhedron
\begin{align*}
	\P = \{ y \in \R^q \mid \exists x \in \R^n :\;y
	\geq M x,\; a \leq Bx \leq b,\;
	 l \leq x \leq u\}\text{.}
\end{align*}
Algorithms for \eqref{molp} like Benson's algorithm 
\cite{benson,benson_type} compute
both a \vrep\ and an \hrep\ of $\P$. A solution to \eqref{molp}
as introduced in \cite{buch_andreas} is closely related to
a V-representation of $\P$, whereas a solution of the dual
problem in the sense of \cite{geom_dual} refers in the same manner to
an \hrep\ of $\P$.

The main idea of computing a \vrep\ and an \hrep\ from a given \prep\ 
$(M,B,a,b,l,u)$ of a polyhedron $A\subseteq \R^q$ is to consider the problem
\begin{align}\label{molp1}\tag{MOLP'}
	\min \begin{pmatrix}
		M\\
		-1^\transpose_{(q)} M
	\end{pmatrix} x \quad \text{s.t.}\quad a \leq B x
	 \leq b,\; l \leq x \leq u 
\end{align}
with upper image
\begin{align*}
	\M \!=\! \{ (y,r) \in \R^q \!\times\! \R |\,
	\exists x \in \R^n\!:\, y \geq M x,\, r \geq
	 -1^\transpose_{(q)} M x,\;   a \leq Bx \leq b,\,
	  l \leq x \leq u\}\text{.}
\end{align*}
It is easily seen that the polyhedron $A$ can be expressed by $\M$ as
\begin{equation} \label{eq:lift}
\begin{aligned} 
	A &= \{ y \in \R^q \mid \exists x \in \R^n:\;
	y = Mx ,\; a \leq Bx \leq b,\;
		 l \leq x \leq u\} \\
	&= \left\{ y \in \R^q \;\bigg|\; \begin{pmatrix}
		y\\
		-1^\transpose_{(q)} y \end{pmatrix} \in \M \right\} 
		\text{.} 
\end{aligned}
\end{equation}

From the \vrep\ and the \hrep\ of $\M$, which are
obtained by solving \eqref{molp1}, one can compute a
\vrep\ and an \hrep\ of the polyhedron $A$. Considering (\ref{eq:lift}) and \cite[Theorem 3]{equivalence_paper}, a \vrep\ of $A$ is obtained from a \vrep\ of $\M$ by deleting directions $d$ with $1^\transpose_{(q+1)} d \neq 0$ and by omitting the $(q+1)$-th components of all remaining vectors.

An \hrep\ of $A$ is easily generated by the \hrep\ of $\mathcal{M}$ employing (\ref{eq:lift}). If the tuple $\left( B^1,a^1,b^1,l^1,u^1 \right)$ is an \hrep\ of $\mathcal{M}$ comprising $2m$ affine inequalities, i.e.\ $B^1 \in \mathbb{R}^{m\times (q+1)}$, then, by replacing the last component of the unknown by the negative sum of the first $q$ components we obtain the \hrep\ of $A$ satisfying the system of inequalities given by
\begin{align*}
\begin{pmatrix}
a^1 \\ l^1_{q+1}
\end{pmatrix} \leq \begin{pmatrix}
B^1_{11}-B^1_{1(q+1)} & \hdots & B^1_{1q}-B^1_{1(q+1)} \\
\vdots & \ddots & \vdots  \\
B^1_{m1}-B^1_{m(q+1)} & \hdots & B^1_{mq}-B^1_{m(q+1)} \\
-1 & \hdots & -1 
\end{pmatrix} y \leq \begin{pmatrix}
b^1 \\ u^1_{q+1} 
\end{pmatrix} ,
\end{align*}
and 
\begin{align*}
 \begin{pmatrix} 
l^1_1 \\ \vdots \\l^1_q 
\end{pmatrix} \leq y\leq \begin{pmatrix}
u^1_1 \\ \vdots \\u^1_q
\end{pmatrix} \, .
\end{align*}     

\section{Numerical experiments}\label{sec:num}
The goal of this section is twofold. First we consider an example from
locational analysis in order to demonstrate how polyhedral calculus
can be used for modeling polyhedral convex optimization problems.
Secondly, we compare our implementation {\em bensolve tools}
\cite{bensolve_tools} with another polyhedral calculus software
by projecting high dimensional polyhedra.

{\em Bensolve tools} is a free and open source software for GNU Octave and
Matlab. It utilizes the VLP solver {\em bensolve} \cite{bensolve}, which is
written in C programming language. The recent version of {\em bensolve tools}
\cite{bensolve_tools} has the following features:
\begin{itemize}
	\item calculus of convex polyhedra,
	\item calculus of polyhedral convex function,
	\item solver for polyhedral convex programs (via LP reformulation),
	\item solver for vector linear programs and
	 multiple objective linear programs ({\em bensolve} interface),
	\item solver for quasi-convace global optimization problems,
	 see \cite{global} for details.
\end{itemize}

\subsection{Polyhedral location problems}

Let a finite number of points $a^1,\ldots,a^m \in \R^n$
be given and let  $d \colon \R^n \times \R^n \rightarrow \R$ be a metric.
We consider the location optimization problem
\begin{align}\label{eq:loc}
\min_{x \in \R^n} \quad \sum_{i=1}^m d(x,a^i) \text{.}
\end{align}
Let $G_i \subseteq \mathbb{R}^n$ be bounded polyhedra with
$ 0 \in \inter G_i$  and let 
$g_i \colon \mathbb{R}^n \rightarrow \mathbb{R}$
be the corresponding {\em gauge function}, which can be defined by
\begin{align*}
 \epi g_i = \cone \left( G_i \times \left\{ 1 \right\}  \right) \text{.}
\end{align*}
Then $g_i$ is a polyhedral convex function and $d(z,y)=g_i(z-y)$ is a metric.

The distance from $x$ to $a^i$ can be expressed by a function
$f_i \colon \mathbb{R}^n \rightarrow \mathbb{R}$, 
$f_i (x)= g_i(x-a^i)$. Its epigraph is
\begin{align*}
\epi f_i = \epi g_i + \begin{pmatrix} a^i \\ 0 \end{pmatrix} \text{.}
\end{align*}
The location problem \eqref{eq:loc} can be written as
\begin{align}\label{eq:loc1}
\min_{x \in \R^n} \sum_{i=1}^m f_i(x) \text{.}
\end{align}
Now it is evident that a \prep\ $(M,B,a,b,l,u)$ of the objective
function $f \colon \mathbb{R}^n \rightarrow \mathbb{R}$,
$f(x) = \sum_{i=1}^m f_i(x)$
can be obtained by the polyhedral calculus
operations discussed above.
Thus, if $(x^*,r^*,z^*)$ is a solution of the linear program
\begin{align*}
	\min_{x,r,z} r \quad \text{s.t.}\quad \begin{pmatrix} x\\r\end{pmatrix}
	=Mz,\; a \leq Bz \leq b,\; l \leq z \leq u \text{,}
\end{align*}
then $x^*$ is an optimal solution of \eqref{eq:loc1}.

\begin{example}
Let $B_1$ be the unit ball of the 1-norm and $B_{\infty}$
be the unit ball of the $\infty$-norm. For all $i$, we set
$G_i = B_1+B_{\infty}$, see 
Figure~\ref{fig:minkowski_sum} in Section~\ref{sec:polyh} for the case $n=3$.
The points $a^i$ are generated randomly on a grid.
The resulting problem instances are solved by \emph{bensolve tools}.
To this end, 
the objective function $f$ is composed from the data and then 
the integrated solver for polyhedral convex programs is used.
The set of {\em all} solutions of \eqref{eq:loc1}
can be obtained in different ways, here it is obtained by
computing the subdifferential of the conjugate of $f$ at $0$
using the corresponding {\em bensolve tools} commands.
The results are illustrated in Figures \ref{im:location_2d} and
\ref{im:location_3d}.
\begin{figure}[hpt]
\includegraphics[width=0.47\textwidth]{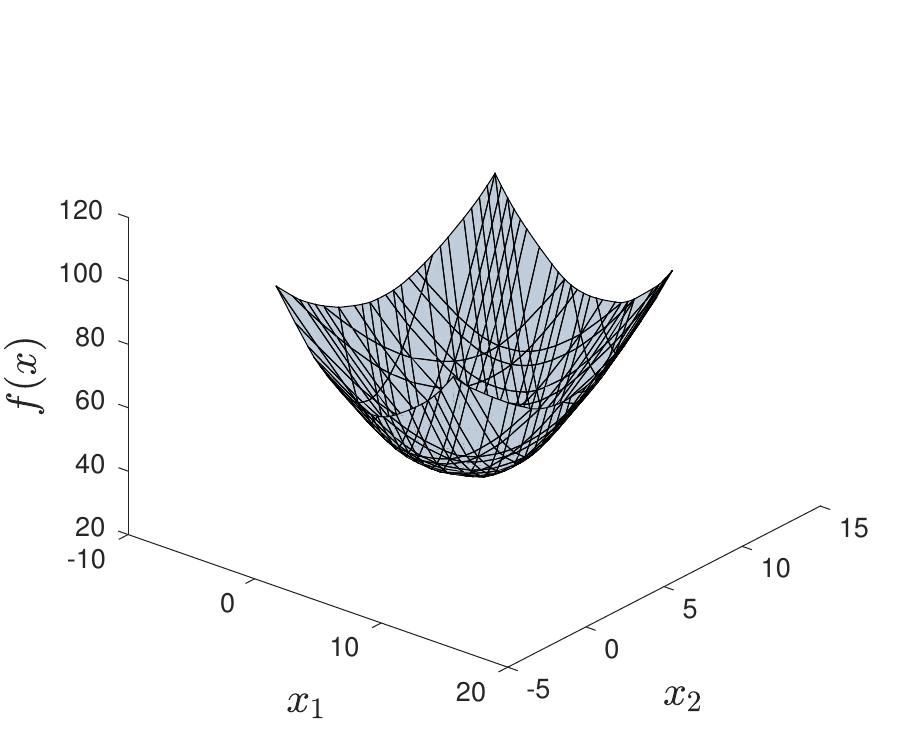}
\includegraphics[width=0.47\textwidth,trim={2cm 8cm 2cm 8cm},clip]{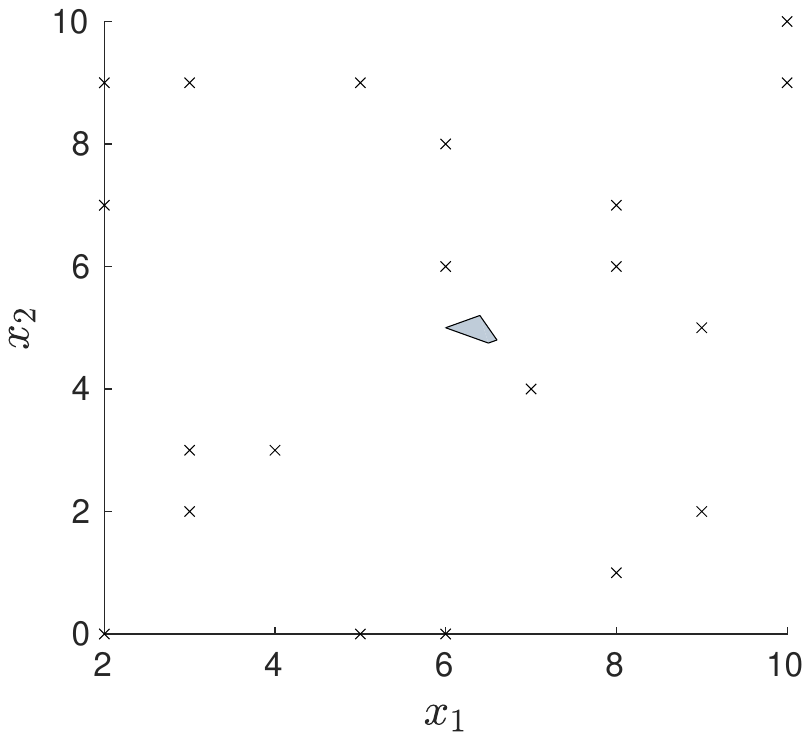}
\caption{Left: the epigraph of $f$; Right: the given points $a^i$ 
and the set of all optimal solutions.}
\label{im:location_2d}
\end{figure}

\begin{figure}[hpt]
\includegraphics[width=0.47\textwidth]{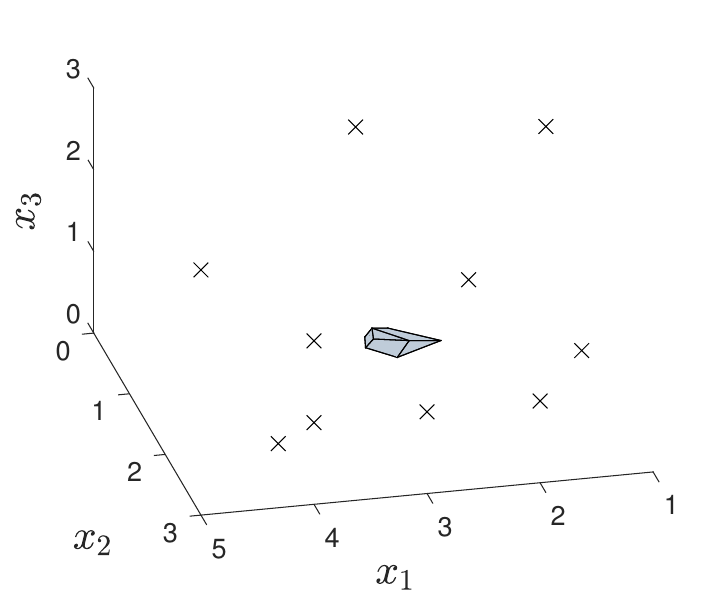}
\includegraphics[width=0.47\textwidth]{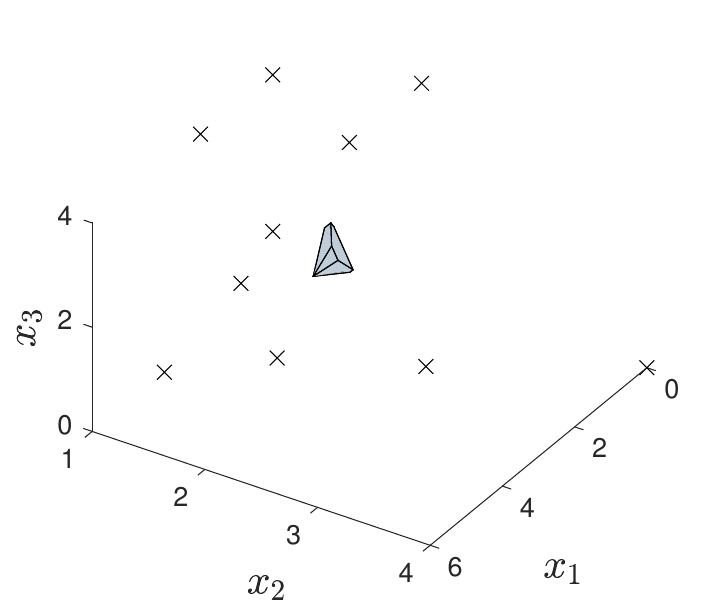}
\caption{The given points $a^i$ 
and the set of all optimal solutions for two instances of the case $n=3$.}
\label{im:location_3d}
\end{figure}
\end{example}

\subsection{Projection of high dimensional polyhedra}

The preceding results show that polyhedral projection is a key tool
for polyhedral calculus. Thus, any projection algorithm computing a
\vrep\ and an \hrep\ can be employed to carry out polyhedral calculus
operations.  While \emph{bensolve tools} uses the projection algorithm
from Section~\ref{sec:proj} which works in the image space $\R^q$,
MPT3 uses a projection technique which works in the variable space
$\R^n$.  The P-represented polyhedra arising in the polyhedral
calculus framework typically encode a projection from a
high-dimensional variable space into a low-dimensional image space,
i.e.\ $q \ll n$.  Thus, with projection algorithms which work in this
low-dimensional image space, one will achieve better results than
using algorithms which work in the high-dimensional variable space
$\R^n$.  In order to support this claim, we compare the performances
of \emph{bensolve tools} and the multi-parametric toolbox MPT3
\cite{MPT3} for randomly generated projection problems in
Example~\ref{ex:proj}.\par
It should be noted that MPT3 has several features such as calculus
operations for nonconvex polyhedra, which are not covered by the
recent version of {\em bensolve tools}.  Moreover, MPT3 provides
different projection algorithms, which might be favorable for other
examples.

\begin{example} \label{ex:proj} Consider an \hrep\
  $(B,a,\emptyset,\emptyset,\emptyset)$ of a polyhedron
  $P = \left\{ x \in \R^n \; | \; Bx \geq a \right\}$ which consists
  of $m=3n$ constraints.  Let the matrix
  $B \in \mathbb{R}^{m \times n}$ consist of uniformly distributed
  (pseudo-)random numbers out of the interval
  $\left[ -1/2 , 1/2 \right] $.  We determine the vector
  $a \in \mathbb{R}^m$ such that the $n$-dimensional simplex $S$ is
  contained in $P$.  This is achieved by setting $a_i$ to the minimum
  of the $i$-th row of the matrix
  $\begin{pmatrix} B , 0_{(n)} \end{pmatrix}$.  We project $P$ onto
  its first $q$ components.  Then, the H- as well as the
  V-representation of the resulting polyhedron are calculated.  Figure
  \ref{fig:num_projection} shows numerical results of different
  instances of this problem computed with \emph{bensolve tools} and
  MPT3.

  \begin{figure}[hpt]
    \includegraphics[width=0.31\textwidth]{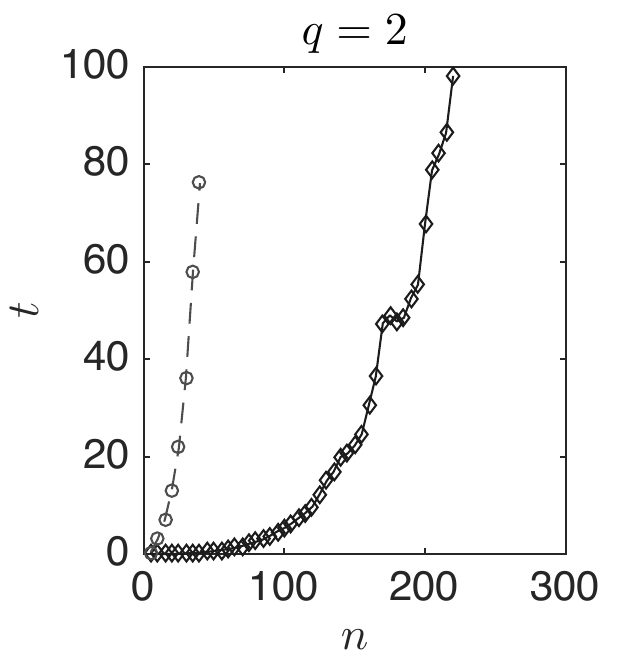}
    \includegraphics[width=0.31\textwidth]{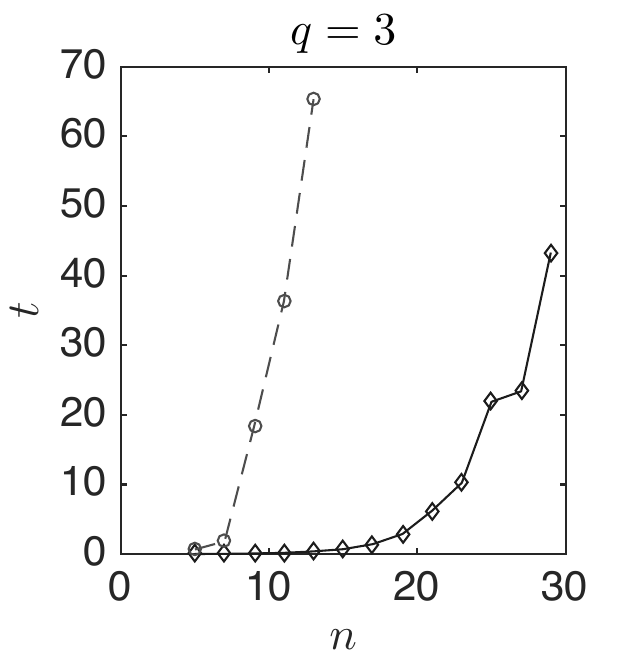}
    \includegraphics[width=0.31\textwidth]{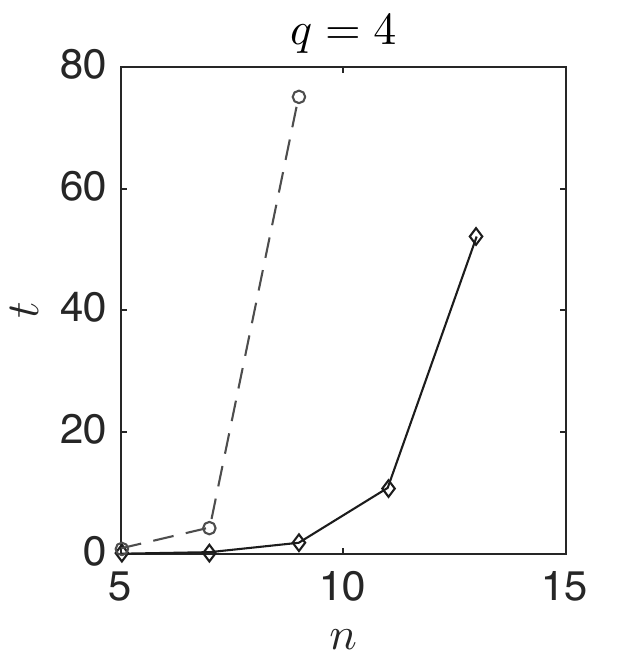}
    \caption{Comparison of {\em bensolve tools} (solid line) and
      MPT3 \cite{MPT3} (dashed line). The average CPU time $t$ in
      seconds of 10 random instances of Example \ref{ex:proj} is displayed.
      The CPU time per instance is limited to 100 seconds. The projection method in
      MPT3 is run with option {\tt 'mplp'} to achieve the best
      performance. The standard options of {\em bensolve tools} are used.}
    \label{fig:num_projection}
  \end{figure}
\end{example}

\section{Conclusions}

In this article, we show how \prep s can be used to perform polyhedral
calculus in an efficient and straight-forward manner. We demonstrate how
the results for polyhedral calculus can be applied to calculus for
polyhedral convex functions. Moreover, we provide methods
and the software {\em bensolve tools} for polyhedral calculus
based on the MOLP solver {\em bensolve}.

%\bibliography{literatur}

\end{document}